\documentclass{amsproc}
\usepackage{amssymb}
\usepackage{graphicx}
\usepackage{euscript}
\usepackage{mathrsfs} 
\usepackage{units}   
\usepackage{color}

\makeatletter
\@namedef{subjclassname@2010}{%
\textup{2010} Mathematics Subject Classification}
\makeatother

\numberwithin{equation}{section}

\newtheorem{thm}{Theorem}[section]
\newtheorem{cor}[thm]{Corollary}
\newtheorem{lem}[thm]{Lemma}
\newtheorem{pro}[thm]{Proposition}

\newtheorem*{thm*}{Theorem}

\theoremstyle{remark}
\newtheorem{rem}[thm]{Remark}

\theoremstyle{definition}

\newtheorem*{dfn*}{Definition}

\DeclareMathOperator{\lin}{\mbox{\sc lin}}
\DeclareMathOperator{\D}{d}
\DeclareMathOperator{\dess}{{\mathsf{Des}}}
\DeclareMathOperator{\dzii}{{\mathsf{Chi}}}

\DeclareMathOperator{\koo}{{\mathsf{root}}}
\DeclareMathOperator{\Koo}{{\mathsf{Root}}}
\DeclareMathOperator{\paa}{{\mathsf{par}}}

\newcommand*{\borel}[1]{{\mathfrak B}(#1)}
\newcommand*{\cbb}{\mathbb C}

\newcommand*{\des}[1]{{\dess(#1)}}
\newcommand*{\dz}[1]{{\EuScript D}(#1)}
\newcommand*{\dzi}[1]{\dzii(#1)}

\newcommand*{\dzn}[1]{{\EuScript D}^\infty(#1)}
\newcommand*{\ee}{\mathcal E}
\newcommand*{\escr}{{\mathscr{E}_V}}
\newcommand*{\ff}{\mathcal F}
\newcommand*{\Ge}{\geqslant}
\newcommand*{\hh}{\mathcal H}

\newcommand*{\is}[2]{\langle#1,#2\rangle}

\newcommand*{\kk}{\mathcal K}
\newcommand*{\Ko}[1]{\Koo(#1)}
\newcommand*{\lambdab}{{\boldsymbol\lambda}}

\newcommand*{\Le}{\leqslant}
\newcommand*{\mm}{\mathscr M}

\newcommand*{\nbb}{\mathbb N}

\newcommand*{\pa}[1]{\paa(#1)}

\newcommand*{\rbb}{\mathbb R}
\newcommand*{\slam}{S_{\boldsymbol \lambda}}
\newcommand*{\smalloplus}{\raise0pt\hbox{$\scriptscriptstyle \oplus$}}

\newcommand*{\tcal}{{\mathscr T}}

\newcommand*{\xx}{\mathcal X}
\newcommand*{\zbb}{\mathbb Z}

\begin{document}
   \title[Unbounded subnormal weighted shifts on directed trees. II]
{Unbounded subnormal weighted shifts on directed
trees. II}
   \author[P.\ Budzy\'{n}ski]{Piotr Budzy\'{n}ski}
   \address{Katedra Zastosowa\'{n} Matematyki,
Uniwersytet Rolniczy w Krakowie, ul.\ Balicka 253c,
PL-30198 Krak\'ow}
   \email{piotr.budzynski@ur.krakow.pl}
   \author[Z.\ J.\ Jab{\l}o\'nski]{Zenon Jan Jab{\l}o\'nski}
\address{Instytut Matematyki, Uniwersytet Jagiello\'nski,
ul.\ \L ojasiewicza 6, PL-30348 Kra\-k\'ow, Poland}
   \email{Zenon.Jablonski@im.uj.edu.pl}
   \author[I.\ B.\ Jung]{Il Bong Jung}
   \address{Department of Mathematics, Kyungpook
National University, Daegu 702-701, Korea}
   \email{ibjung@knu.ac.kr}
   \author[J.\ Stochel]{Jan Stochel}
\address{Instytut Matematyki, Uniwersytet Jagiello\'nski,
ul.\ \L ojasiewicza 6, PL-30348 Kra\-k\'ow, Poland}
   \email{Jan.Stochel@im.uj.edu.pl}
   \thanks{The research of
the first, second and fourth authors was supported by
the MNiSzW (Ministry of Science and Higher Education)
grant NN201 546438 (2010-2013). The research of the
first author was partially supported by the NCN
(National Science Center) grant
DEC-2011/01/D/ST1/05805. The third author was
supported by Basic Science Research Program through
the National Research Foundation of Korea (NRF) funded
by the Ministry of Education, Science and Technology
(2009-0087565).}
   \begin{abstract}
Criteria for subnormality of unbounded injective
weighted shifts on leafless directed trees with one
branching vertex are proposed. The case of classical
weighted shifts is discussed. The relevance of an
inductive limit approach to subnormality of weighted
shifts on directed trees is revealed.
   \end{abstract}
   \maketitle
   \section{Introduction}
This paper is concerned with unbounded subnormal
weighted shifts on directed trees (see \cite{j-j-s}
for basic facts on weighted shifts on directed trees
and \cite{b-j-j-sA} for the literature on
subnormality). The only known general
characterizations of subnormality of unbounded Hilbert
space operators are due to Bishop and Foia\c{s}
\cite{bis,foi}, and Szafraniec \cite{FHSz}. These
characterizations refer to either semispectral
measures or elementary spectral measures. They seem to
be difficult to apply in the context of weighted
shifts on directed trees, especially when we want them
to be formulated in terms of weights. The situation is
much better when the operators in question have
invariant domains. For such operators the full
characterization of subnormality has been given in
\cite{StSz1} (see also \cite{c-s-sz} for a new
simplifying approach). Using this abstract tool, we
have invented in \cite[Theorem 5.1.1]{b-j-j-sA} a new
criterion (read:\ a sufficient condition) for
subnormality of weighted shifts on directed trees
written in terms of consistent systems of probability
measures. The main objective of the present paper is
to develop this idea in the context of directed trees
$\tcal_{\eta,\kappa}$ which are models of leafless
directed trees with one branching vertex (see Section
\ref{obv}). The characterizations of subnormality of
bounded weighted shifts on $\tcal_{\eta,\kappa}$ given
in \cite[Corollary 6.2.2]{j-j-s} (see also
\cite[Theorem 6.2.1]{j-j-s}) have been deduced from
the celebrated Lambert's theorem (cf.\ \cite{Lam})
which is no longer valid for unbounded operators (even
for weighted shifts on $\tcal_{\infty,\kappa}$, cf.\
\cite{j-j-s4}). However, as proved in Theorem
\ref{omega2}, all but one criteria in \cite[Corollary
6.2.2]{j-j-s} remain valid in the unbounded case. The
exception is essentially different (cf.\ Theorem
\ref{omega2}). Under the assumption of determinacy,
the criteria in Theorem \ref{omega2} become full
characterizations (see Theorem \ref{deter}). An
inductive limit approach to subnormality is shown to
be applicable in the context of weighted shifts on
directed trees (see Section \ref{sec-in}).

This paper is a sequel to \cite{b-j-j-sA}, to which we
refer readers for more background, discussion and
references.
   \section{Preliminaries}
Let $\zbb$ and $\cbb$ stand for the sets of integers
and complex numbers respectively. Denote by $\rbb_+$
the set of all nonnegative real numbers. Set $\zbb_+ =
\{0,1,2,3,\ldots\}$ and $\nbb = \{1,2,3,4,\ldots\}$.
We write $\borel{\rbb_+}$ for the $\sigma$-algebra of
all Borel subsets of $\rbb_+$ and $\delta_0$ for the
Borel probability measure on $\rbb_+$ concentrated at
$0$. A sequence $\{t_n\}_{n=0}^\infty \subseteq
\rbb_+$ is said to be a {\em Stieltjes moment
sequence} if there exists a positive Borel measure
$\mu$ on $\rbb_+$ such that $t_{n}=\int_0^\infty s^n
\D\mu(s)$ for every $n\in \zbb_+$, where
$\int_0^\infty$ means the integral over $\rbb_+$; such
$\mu$ is called a {\em representing measure} of
$\{t_n\}_{n=0}^\infty$. We say that a Stieltjes moment
sequence is {\em determinate} if it has only one
representing measure. A two-sided sequence
$\{t_n\}_{n=-\infty}^\infty\subseteq \rbb_+$ is said
to be a {\em two-sided Stieltjes moment sequence} if
there exists a positive Borel measure $\mu$ on
$(0,\infty)$ such that $t_{n}=\int_{(0,\infty)} s^n
\D\mu(s)$ for every $n \in \zbb$; such $\mu$ is called
a {\em representing measure} of
$\{t_n\}_{n=-\infty}^\infty$. It follows from
\cite[page 202]{ber} (see also \cite[Theorem
6.3]{j-t-w}) that
   \begin{align} \label{char2sid}
   \begin{minipage}{29em}
$\{t_n\}_{n=-\infty}^\infty \subseteq \rbb_+$ is a
two-sided Stieltjes moment sequence if and only if
$\{t_{n-k}\}_{n=0}^\infty$ is a Stieltjes moment
sequence for every $k \in \zbb_+$.
   \end{minipage}
   \end{align}
We refer the reader to \cite{ber,sim} for the
foundations of the theory of moment problems.

Let $A$ be an operator in a complex Hilbert space
$\hh$ (all operators considered in this paper are
linear). Denote by $\dz{A}$ the domain of $A$. Set
$\dzn{A} = \bigcap_{n=0}^\infty\dz{A^n}$. A linear
subspace $\ee$ of $\dz{A}$ is said to be a {\em core}
of $A$ if the graph of $A$ is contained in the closure
of the graph of the restriction $A|_{\ee}$ of $A$ to
$\ee$. A densely defined operator $S$ in $\hh$ is said
to be {\em subnormal} if there exists a complex
Hilbert space $\kk$ and a normal operator $N$ in $\kk$
such that $\hh \subseteq \kk$ (isometric embedding)
and $Sh = Nh$ for all $h \in \dz S$. We refer the
reader to \cite{b-s,weid} for background on unbounded
operators. We write $\lin \ff$ for the linear span of
a subset $\ff$ of $\hh$.

Let $\tcal=(V,E)$ be a directed tree ($V$ and $E$
stand for the sets of vertices and edges of $\tcal$,
respectively). Denote by $\koo$ the root of $\tcal$
(provided it exists) and write $\Ko{\tcal}=\{\koo\}$
if $\tcal$ has a root and $\Ko{\tcal} = \varnothing$
otherwise. Define $V^\circ=V\setminus \Ko{\tcal}$. Set
$\dzi u = \{v\in V\colon (u,v)\in E\}$ for $u \in V$.
A member of $\dzi u$ is called a {\em child} of $u$.
Denote by $\paa$ the partial function from $V$ to $V$
which assigns to each vertex $u\in V^\circ$ its parent
$\pa{u}$ (i.e.\ a unique $v \in V$ such that $(v,u)\in
E$). Set $\des u = \bigcup_{n=0}^\infty \{w \in
V\colon \paa^n(w)=u\}$ for $u\in V$. Note that the
terms in the union are pairwise disjoint. We refer the
reader to \cite{b-j-j-sA,j-j-s} for all facts about
directed trees needed in this paper.

Denote by $\ell^2(V)$ the Hilbert space of all square
summable complex functions on $V$ with the inner
product $\is fg = \sum_{u \in V} f(u)
\overline{g(u)}$. For $u \in V$, we define $e_u \in
\ell^2(V)$ to be the characteristic function of the
one-point set $\{u\}$. Then $\{e_u\}_{u\in V}$ is an
orthonormal basis of $\ell^2(V)$. Set $\escr = \lin
\{e_u\colon u \in V\}$. By a {\em weighted shift} on
$\tcal$ with weights $\lambdab=\{\lambda_v\}_{v \in
V^\circ} \subseteq \cbb$ we mean the operator $\slam$
in $\ell^2(V)$ defined by
   \begin{align*}
   \begin{aligned}
\dz {\slam} & = \{f \in \ell^2(V) \colon
\varLambda_\tcal f \in \ell^2(V)\},
   \\
\slam f & = \varLambda_\tcal f, \quad f \in \dz
{\slam},
   \end{aligned}
   \end{align*}
where $\varLambda_\tcal$ is the mapping defined on
functions $f\colon V \to \cbb$ via
   \begin{align} \label{lamtauf}
(\varLambda_\tcal f) (v) =
   \begin{cases}
\lambda_v \cdot f\big(\pa v\big) & \text{ if } v\in
V^\circ,
   \\
0 & \text{ if } v=\koo.
   \end{cases}
   \end{align}
Suppose $\escr \subseteq \dzn{\slam}$ and $u \in V$ is
such that $\dzi{u}\neq \varnothing$ and $\{\|\slam^n
e_v\|^2\}_{n=0}^\infty$ is a Stieltjes moment sequence
with a representing measure $\mu_v$ for every $v \in
\dzi u$. Then, following \cite{j-j-s,b-j-j-sA}, we say
that $\slam$ satisfies the {\em consistency condition}
at $u$ if\footnote{\;We adhere to the standard
convention that $0 \cdot \infty = 0$.}
   \begin{align} \label{alanconsi}
\sum_{v \in \dzi{u}} |\lambda_v|^2 \int_0^\infty \frac
1 s\, \D \mu_v(s) \Le 1,
   \end{align}

   Let us recall the main result of \cite{b-j-j-sA}.
   \begin{thm}[\mbox{\cite[Theorem 5.1.1]{b-j-j-sA}}]\label{main}
Assume that $\escr \subseteq \dzn{\slam}$ and there
exist a system $\{\mu_v\}_{v \in V}$ of Borel
probability measures on $\rbb_+$ and a system
$\{\varepsilon_v\}_{v \in V}$ of real
numbers\footnote{\;Note that \eqref{muu+} implies that
the system $\{\varepsilon_v\}_{v \in V}$ consists of
nonnegative real numbers.} that satisfy the following
condition
   \begin{align} \label{muu+}
\mu_u(\sigma) = \sum_{v \in \dzi u} |\lambda_v|^2
\int_\sigma \frac 1 s \D \mu_v(s) + \varepsilon_u
\delta_0(\sigma), \quad \sigma \in \borel{\rbb_+}, \,
u \in V.
      \end{align}
Then $\slam$ is subnormal.
   \end{thm}
   \begin{dfn*}
We say that $u \in V$ is a {\em Stieltjes vertex}
(with respect to $\slam$) if $e_u \in \dzn{\slam}$ and
$\{\|\slam^n e_u\|^2\}_{n=0}^\infty$ is a Stieltjes
moment sequence.
   \end{dfn*}
In many cases, if each child of $u$ is a Stieltjes
vertex, then so is $u$ itself (cf.\ \cite[Lemma
4.1.3]{b-j-j-sA}). If $u$ is a Stieltjes vertex, then
in general its children are not (cf.\ \cite[Example
6.1.6]{j-j-s}). However, assuming that $u$ has only
one child $w$ and $\lambda_w \neq 0$, if $u$ is a
Stieltjes vertex, then so is $w$ (see \cite[Lemma
6.1.5]{j-j-s} for the bounded case).
   \begin{lem} \label{charsub-1}
Let $\slam$ be a weighted shift on $\tcal$ with
weights $\lambdab = \{\lambda_v\}_{v \in V^\circ}$ and
let $u, w \in V$ be such that $\dzi{u} = \{w\}$.
Suppose $u$ is a Stieltjes vertex and $\lambda_{w}\neq
0$. Then $w$ is a Stieltjes vertex and the following
two assertions hold\/{\em :}
   \begin{enumerate}
   \item[(i)] the mapping
$\mm_{w}^{\mathrm b}(\lambdab) \ni \mu \to \rho_{\mu}
\in \mm_{u}(\lambdab)$ defined by
   \begin{align*}
\rho_{\mu}(\sigma) = |\lambda_{w}|^2 \int_\sigma \frac
1 s \D \mu(s) + \Big(1 - |\lambda_{w}|^2 \int_0^\infty
\frac 1 s \D \mu(s)\Big) \delta_0(\sigma), \quad
\sigma \in \borel{\rbb_+},
   \end{align*}
is a bijection with the inverse $\mm_{u}(\lambdab) \ni
\rho \to \mu_{\rho} \in \mm_{w}^{\mathrm b}(\lambdab)$
given by
   \begin{align*}
\mu_{\rho} ( \sigma) = \frac 1 {|\lambda_{w}|^2}
\int_\sigma s \D \rho (s), \quad \sigma \in
\borel{\rbb_+},
   \end{align*}
where $\mm_{w}^{\mathrm b}(\lambdab)$ is the set of
all representing measures $\mu$ of $\{\|\slam^n
e_{w}\|^2\}_{n=0}^\infty$ such that $\int_0^\infty
\frac 1 s \D \mu(s) \Le \frac{1}{|\lambda_{w}|^2}$,
and $\mm_{u}(\lambdab)$ is the set of all representing
measures $\rho$ of $\{\|\slam^n
e_{u}\|^2\}_{n=0}^\infty$,
   \item[(ii)] if the Stieltjes moment
sequence $\{\|\slam^n e_{w}\|^2\}_{n=0}^\infty$ is
determinate, then so are $\{\|\slam^n
e_{u}\|^2\}_{n=0}^\infty$ and $\{\|\slam^{n+1}
e_{u}\|^2\}_{n=0}^\infty$.
   \end{enumerate}
   \end{lem}
   \begin{proof}  Since $e_{u} \in
\dzn{\slam}$, $\dzi{u} = \{w\}$ and $\lambda_{w} \neq
0$, we infer from \cite[Proposition 3.1.3]{j-j-s} that
$e_{w} = \frac{1}{\lambda_{w}} \slam e_{u} \in
\dzn{\slam}$, and thus
  \begin{align*}
\|\slam^n e_{w}\|^2 = \frac 1
{|\lambda_{w}|^2}\|\slam^{n+1} e_{u}\|^2, \quad n \in
\zbb_+.
   \end{align*}
This and \cite[Lemma 2.4.1]{b-j-j-sA} with
$\vartheta=1$ and $t_n=\|\slam^{n+1} e_{u}\|^2$
complete the proof.
   \end{proof}
   \section{\label{cws}Classical weighted shifts}
In this section we focus on classical weighted shifts
(see \cite{shi,ml} for both the bounded and the
unbounded cases). As shown in \cite[Remark
3.1.4]{j-j-s}, unilateral and bilateral weighted shifts
can be regarded as weighted shifts on directed trees
\mbox{$(\zbb_+, \{(n,n+1)\colon n \in \zbb_+\})$} and
\mbox{$(\zbb, \{(n,n+1)\colon n \in \zbb\})$}
respectively. The reader should be aware that we
enumerate weights of unilateral and bilateral weighted
shifts in accordance with our notation, i.e.,
   \begin{align} \label{notnew}
\slam e_n = \lambda_{n+1} e_{n+1}, \quad n\in \zbb_+
\;\; (\textrm{respectively:\ } n \in \zbb).
   \end{align}

Using our approach, we can derive the
Berger-Gellar-Wallen criterion for subnormality of
injective unilateral weighted shifts (see
\cite{g-w,hal2} for the bound\-ed case and
\cite[Theorem 4]{StSz1} for the unbounded one).
   \begin{thm} \label{b-g-w}
If $\slam$ is a unilateral weighted shift with nonzero
weights $\lambdab = \{\lambda_n\}_{n=1}^\infty$
$($cf.\ \eqref{notnew}$)$, then the following three
conditions are equivalent\/{\em :}
   \begin{enumerate}
   \item[(i)] $\slam$ is subnormal,
   \item[(ii)]  $(1, |\lambda_1|^2, |\lambda_1
\lambda_2|^2, |\lambda_1 \lambda_2 \lambda_3|^2,
\ldots)$ is a Stieltjes moment sequence,
   \item[(iii)]
$k$ is a Stieltjes vertex for every $k \in \zbb_+$.
   \end{enumerate}
   \end{thm}
   \begin{proof}
It is clear that $\escr \subseteq \dzn{\slam}$.

(i)$\Rightarrow$(iii) Employ \cite[Proposition
4.1.1]{b-j-j-sA}.

(iii)$\Rightarrow$(ii) This is evident, because the
sequence $(1, |\lambda_1|^2, |\lambda_1 \lambda_2|^2,
|\lambda_1 \lambda_2 \lambda_3|^2, \ldots)$ coincides
with $\{\|\slam^n e_0\|^2\}_{n=0}^\infty$.

(ii)$\Rightarrow$(i) Let $\mu$ be a representing measure of
the Stieltjes moment sequence $\{\|\slam^n
e_0\|^2\}_{n=0}^\infty$ (which in general may not be
determinate, cf.\ \cite{sz3}). Define the sequence
$\{\mu_n\}_{n=0}^\infty$ of Borel probability measures on
$\rbb_+$ by
   \begin{align*}
\mu_n(\sigma) = \frac{1}{\|\slam^n e_0\|^2}
\int_{\sigma} s^n \D \mu(s), \quad \sigma \in
\borel{\rbb_+}, \, n \in \zbb_+.
   \end{align*}
It is then clear that
   \begin{align*}
\mu_0(\sigma) &=|\lambda_{1}|^2 \int_\sigma
\frac{1}{s} \D\mu_{1}(s) + \mu(\{0\}) \delta_0
(\sigma), \quad \sigma \in \borel{\rbb_+},
   \\
\mu_n(\sigma) &= |\lambda_{n+1}|^2 \int_\sigma
\frac{1}{s} \D\mu_{n+1}(s), \quad \sigma \in
\borel{\rbb_+},\, n \in \nbb,
   \end{align*}
which means that the systems $\{\mu_n\}_{n=0}^\infty$
and $\{\varepsilon_n\}_{n=0}^\infty := (\mu(\{0\}), 0,
0, \ldots)$ satisfy the assumptions of Theorem
\ref{main}. This completes the proof.
   \end{proof}
Now we prove an analogue of the Berger-Gellar-Wallen
criterion for subnormality of injective bilateral
weighted shifts (see \cite[Theorem II.6.12]{con2} for
the bounded case and \cite[Theorem 5]{StSz1} for the
unbounded one).
   \begin{thm} \label{b-g-w-2}
If $\slam$ is a bilateral weighted shift with nonzero
weights $\lambdab=\{\lambda_n\}_{n \in \zbb}$ $($cf.\
\eqref{notnew}$)$, then the following four conditions
are equivalent\/{\em :}
   \begin{enumerate}
   \item[(i)] $\slam$ is subnormal,
   \item[(ii)] the two-sided sequence $\{t_n\}_{n=-\infty}^\infty$
defined by
   \begin{align*}
t_n =
   \begin{cases}
|\lambda_1 \cdots \lambda_{n}|^2 & \text{ for } n \Ge
1,
   \\
1 & \text{ for } n=0,
   \\
|\lambda_{n+1} \cdots \lambda_{0}|^{-2} & \text{ for }
n \Le -1,
   \end{cases}
   \end{align*}
is a two-sided Stieltjes moment sequence,
   \item[(iii)]
$-k$ is a Stieltjes vertex for infinitely many
nonnegative integers $k$,
   \item[(iv)]
$k$ is a Stieltjes vertex for every $k \in \zbb$.
   \end{enumerate}
   \end{thm}
   \begin{proof}
Clearly, $\escr \subseteq \dzn{\slam}$.

(i)$\Rightarrow$(iv) Employ \cite[Proposition
4.1.1]{b-j-j-sA}.

(iv)$\Rightarrow$(iii) Evident.

(iii)$\Rightarrow$(iv) Apply Lemma \ref{charsub-1}.

(iv)$\Rightarrow$(ii) Since $t_{n-k} = t_{-k}
\|\slam^n e_{-k}\|^2$ for all $n \in \zbb$ and $k\in
\zbb_+$, we can apply the characterization
\eqref{char2sid}.

(ii)$\Rightarrow$(i) Let $\mu$ be a representing
measure of $\{t_n\}_{n=-\infty}^\infty$. Define the
two-sided sequence $\{\mu_n\}_{n=-\infty}^\infty$ of
Borel probability measures on $\rbb_+$ by (note that
$\mu(\{0\})=0$)
   \begin{align*}
\mu_n(\sigma) = \frac{1}{\|\slam^n e_0\|^2}
\int_{\sigma} s^n \D \mu(s), \quad \sigma \in
\borel{\rbb_+}, \, n \in \zbb.
   \end{align*}
We easily verify that
   \begin{align*}
\mu_n(\sigma) &= |\lambda_{n+1}|^2 \int_\sigma
\frac{1}{s} \D\mu_{n+1}(s), \quad \sigma \in
\borel{\rbb_+},\, n \in \zbb,
   \end{align*}
which means that the systems
$\{\mu_n\}_{n=-\infty}^\infty$ and
$\{\varepsilon_n\}_{n=-\infty}^\infty$ with
$\varepsilon_n\equiv 0$ satisfy the assumptions of
Theorem \ref{main}. This completes the proof.
   \end{proof}
In view of Theorems \ref{b-g-w} and \ref{b-g-w-2}, the
necessary condition for subnormality of general
operators given in \cite[Proposition 3.2.1]{b-j-j-sA}
becomes sufficient in the case of injective classical
weighted shifts. To the best of our knowledge, the
class of injective classical weighted shifts seems to
be the only one\footnote{\;Not mentioning symmetric
operators which are always subnormal, cf.\
\cite[Appendix I.2]{a-g}.} for which this phenomenon
occurs regardless of whether or not the operators in
question have sufficiently many qusi-analytic vectors
(see \cite{StSz0} for more details; see also
\cite[Sections 3.2 and 5.3]{b-j-j-sA}).
   \section{\label{obv}One branching vertex} Our next
aim is to discuss subnormality of weighted shifts with
nonzero weights on directed trees that have only one
{\em branching vertex}, i.e., a vertex with at least
two children. By \cite[Proposition 5.2.1]{b-j-j-sA},
there is no loss of generality in assuming that the
directed tree under consideration is countably
infinite and leafless. Such directed trees are one
step more complicated than those involved in the
definitions of classical weighted shifts (see Section
\ref{cws}). Countably infinite and leafless directed
trees with one branching vertex can be modelled as
follows (see Figure 1). Given $\eta,\kappa \in \zbb_+
\sqcup \{\infty\}$ with $\eta \Ge 2$, we define the
directed tree $\tcal_{\eta,\kappa} = (V_{\eta,\kappa},
E_{\eta,\kappa})$ by\footnote{\;The symbol
``\,$\sqcup$\,'' denotes disjoint union of sets.}
   \allowdisplaybreaks
   \begin{align*}
   \begin{aligned}
V_{\eta,\kappa} & = \big\{-k\colon k\in J_\kappa\big\}
\sqcup \{0\} \sqcup \big\{(i,j)\colon i\in J_\eta,\,
j\in \nbb\big\},
   \\
E_{\eta,\kappa} & = E_\kappa \sqcup
\big\{(0,(i,1))\colon i \in J_\eta\big\} \sqcup
\big\{((i,j),(i,j+1))\colon i\in J_\eta,\, j\in
\nbb\big\},
   \\
E_\kappa & = \big\{(-k,-k+1) \colon k\in
J_\kappa\big\},
   \end{aligned}
   \end{align*}
where $J_\iota := \{k \in \nbb\colon k\Le \iota\}$ for
$\iota \in \zbb_+ \sqcup \{\infty\}$.
   \vspace{1.5ex}
   \begin{center}
   \includegraphics[width=7cm]
{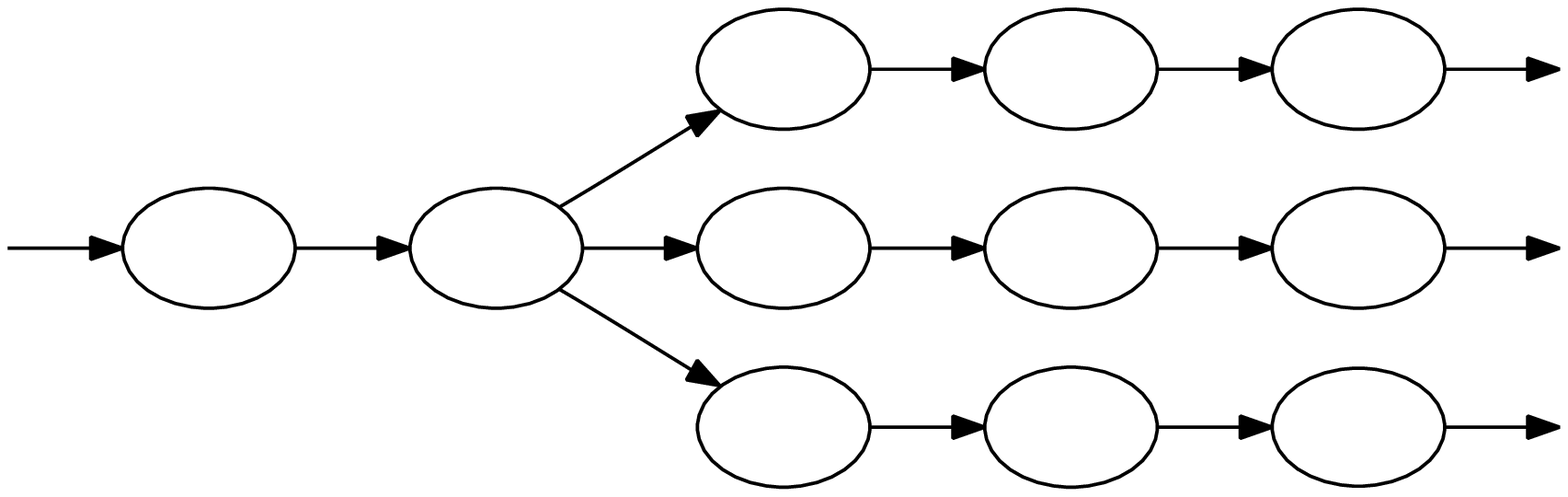}
   \\[1.5ex]
{\small {\sf Figure 1}}
   \end{center}
   \vspace{1ex}
   If $\kappa < \infty$, then the directed tree
$\tcal_{\eta,\kappa}$ has the root $-\kappa$. If
$\kappa=\infty$, then $\tcal_{\eta,\infty}$ is
rootless. In all cases, $0$ is the branching vertex of
$\tcal_{\eta,\kappa}$.

{\em Caution.} One of the advantages of considering
weighted shifts $\slam$ on $\tcal_{\eta,\kappa}$ is
that $e_0\in \dzn{\slam}$ if and only if
$\mathscr{E}_{V_{\eta,\kappa}} \subseteq \dzn{\slam}$.

   We begin by showing that all but one criteria for
subnormality of injective weighted shifts on
$\tcal_{\eta,\kappa}$ given in \cite[Corollary
6.2.2]{j-j-s} remain valid in the unbounded case. The
only exception is condition (iii) in Theorem
\ref{omega2} below which is an unbounded variant
(essentially different) of condition \mbox{(ii-b)} in
\cite[Corollary 6.2.2]{j-j-s}. Below, we adhere to the
notation $\lambda_{i,j}$ instead of a more formal
expression $\lambda_{(i,j)}$.
   \begin{thm}\label{omega2}
Let $\slam$ be a weighted shift on
$\tcal_{\eta,\kappa}$ with nonzero weights $\lambdab =
\{\lambda_v\}_{v \in V_{\eta,\kappa}^\circ}$ such that
$e_0 \in \dzn{\slam}$. Suppose that there exists a
sequence $\{\mu_i\}_{i=1}^\eta$ of Borel probability
measures on $\rbb_+$ such that
   \begin{align} \label{zgod0}
\int_0^\infty s^n \D \mu_i(s) =
\Big|\prod_{j=2}^{n+1}\lambda_{i,j}\Big|^2, \quad n
\in \nbb, \; i \in J_\eta.
   \end{align}
Then $\slam$ is subnormal provided one of the
following four conditions holds\/{\em :}
   \begin{enumerate}
   \item[(i)]  $\kappa=0$ and
   \begin{align}  \label{zgod}
\sum_{i=1}^\eta |\lambda_{i,1}|^2 \int_0^\infty \frac
1 s\, \D \mu_i(s) \Le 1,
   \end{align}
   \item[(ii)] $0 < \kappa < \infty$ and
   \begin{align} \label{zgod'}
\sum_{i=1}^\eta |\lambda_{i,1}|^2 \int_0^\infty \frac
1 s\, \D \mu_i(s) &= 1,
   \\
\Big|\prod_{j=0}^{l-1} \lambda_{-j}\Big|^2
\sum_{i=1}^\eta|\lambda_{i,1}|^2 \int_0^\infty \frac 1
{s^{l+1}} \D \mu_i(s) & = 1, \quad l \in J_{\kappa-1},
\label{widly1}
   \\
\Big|\prod_{j=0}^{\kappa-1}
\lambda_{-j}\Big|^2\sum_{i=1}^\eta|\lambda_{i,1}|^2
\int_0^\infty \frac 1 {s^{\kappa+1}} \D \mu_i(s) & \Le
1, \label{widly1'}
   \end{align}
   \item[(iii)] $0 < \kappa < \infty$ and there exists
a Borel probability measure $\nu$ on $\rbb_+$ such
that
   \begin{align} \label{prob}
\int_0^\infty s^n \D \nu(s) & =
\Big|\prod_{j=\kappa-n}^{\kappa-1}\lambda_{-j}\Big|^2,
\quad n \in J_\kappa,
   \\     \label{prob'}
\int_\sigma s^\kappa \D \nu(s) & =
\Big|\prod_{j=0}^{\kappa-1} \lambda_{-j}\Big|^2 \;
\sum_{i=1}^\eta |\lambda_{i,1}|^2 \int_\sigma
\frac{1}{s} \D \mu_i(s), \quad \sigma \in
\borel{\rbb_+},
   \end{align}
   \item[(iv)] $\kappa=\infty$ and equalities \eqref{zgod'}
and \eqref{widly1} are satisfied.
   \end{enumerate}
   \end{thm}
   \begin{proof}
(i) Define the system of Borel probability measures
$\{\mu_v\}_{v\in V_{\eta,0}}$ on $\rbb_+$ and the
system $\{\varepsilon_v\}_{v\in V_{\eta,0}}$ of
nonnegative real numbers by
   \begin{align*}
\mu_{0}(\sigma) & = \sum_{i=1}^{\eta} |\lambda_{i,1}|^2
\int_\sigma \frac 1 s \D \mu_i(s) + \varepsilon_0
\delta_0(\sigma), \quad \sigma \in \borel{\rbb_+},
   \\
\varepsilon_0 & = 1 - \sum_{i=1}^\eta
|\lambda_{i,1}|^2 \int_0^\infty \frac 1 s\, \D
\mu_i(s),
   \end{align*}
and
   \begin{equation} \label{kap0}
\left\{
   \begin{aligned}
\mu_{i,n}(\sigma) & = \frac{1}{\|\slam^{n-1}
e_{i,1}\|^2} \int_{\sigma} s^{n-1} \D \mu_i(s), \quad
\sigma \in \borel{\rbb_+}, \, i \in J_\eta, \, n \in
\nbb,
   \\[1ex]
\varepsilon_{i,n} & = 0, \quad i \in J_\eta, \, n \in
\nbb.
   \end{aligned}
\right.
   \end{equation}
(We write $\mu_{i,j}$ and $\varepsilon_{i,j}$ instead
of the more formal expressions $\mu_{(i,j)}$ and
$\varepsilon_{(i,j)}$.) Clearly $\mu_{i,1}=\mu_i$ for
all $i \in J_\eta$. Using \eqref{zgod0} and
\eqref{zgod}, we verify that the systems
$\{\mu_v\}_{v\in V_{\eta,0}}$ and
$\{\varepsilon_v\}_{v\in V_{\eta,0}}$ are well-defined
and satisfy the assumptions of Theorem \ref{main}.
Hence $\slam$ is subnormal.

(ii) Define the systems $\{\mu_v\}_{v\in
V_{\eta,\kappa}}$ and $\{\varepsilon_v\}_{v\in
V_{\eta,\kappa}}$ by \eqref{kap0} and
   \begin{align}\label{literki1}
\mu_{0}(\sigma) & = \sum_{i = 1}^{\eta}
|\lambda_{i,1}|^2 \int_\sigma \frac 1 s \D \mu_i(s),
\quad \sigma \in \borel{\rbb_+},
   \\   \label{literki2}
   \mu_{-l} (\sigma) & = \Big|\prod_{j=0}^{l-1}
\lambda_{-j}\Big|^2 \sum_{i=1}^{\eta} |\lambda_{i,1}|^2
\int_{\sigma} \frac 1 {s^{l+1}}\, \D \mu_i(s), \quad
\sigma \in \borel{\rbb_+}, \, l\in J_{\kappa-1},
   \\ \label{literki3}
   \mu_{-\kappa} (\sigma) & =
\Big|\prod_{j=0}^{\kappa-1} \lambda_{-j}\Big|^2
\sum_{i=1}^{\eta} |\lambda_{i,1}|^2 \int_{\sigma} \frac
1 {s^{\kappa+1}}\, \D \mu_i(s) + \varepsilon_{-\kappa}
\delta_0(\sigma), \hspace{0.8ex} \sigma \in
\borel{\rbb_+}, \
   \\  \label{literki4}
\varepsilon_v & =
   \begin{cases}
0 & \text{ if } v\in V_{\eta,\kappa}^\circ,
   \\
   1 - \Big|\prod_{j=0}^{\kappa-1}
\lambda_{-j}\Big|^2\sum_{i=1}^\eta|\lambda_{i,1}|^2
\int_0^\infty \frac 1 {s^{\kappa+1}} \D \mu_i(s) &
\text{ if } v = - \kappa.
   \end{cases}
   \end{align}
Applying \eqref{zgod0}, \eqref{zgod'}, \eqref{widly1}
and \eqref{widly1'}, we verify that the systems
$\{\mu_v\}_{v\in V_{\eta,\kappa}}$ and
$\{\varepsilon_v\}_{v\in V_{\eta,\kappa}}$ are
well-defined and satisfy the assumptions of Theorem
\ref{main}. Therefore $\slam$ is subnormal.

(iii) First note that $\|\slam^n e_{-\kappa}\|^2 =
\Big|\prod_{j=\kappa-n}^{\kappa-1}\lambda_{-j}\Big|^2$
for $n \in J_\kappa$. Define the systems
$\{\mu_v\}_{v\in V_{\eta,\kappa}}$ and
$\{\varepsilon_v\}_{v\in V_{\eta,\kappa}}$ by
\eqref{kap0} and
   \begin{align*}
\mu_{-l}(\sigma) & = \frac{1}{\|\slam^{-l + \kappa}
e_{-\kappa}\|^2}\int_\sigma s^{-l + \kappa} \D\nu(s),
\quad \sigma \in \borel{\rbb_+},\, l\in J_\kappa \cup
\{0\},
   \\
\varepsilon_v & =
   \begin{cases}
0 & \text{ if } v\in V_{\eta,\kappa}^\circ,
   \\
\nu(\{0\}) & \text{ if } v = - \kappa.
   \end{cases}
   \end{align*}
Clearly $\mu_{-\kappa}=\nu$, which together with
\eqref{zgod0}, \eqref{prob} and \eqref{prob'} implies
that the systems $\{\mu_v\}_{v\in V_{\eta,\kappa}}$
and $\{\varepsilon_v\}_{v\in V_{\eta,\kappa}}$ satisfy
the assumptions of Theorem \ref{main}. As a
consequence, $\slam$ is subnormal.

(iv) Define the system $\{\mu_v\}_{v\in
V_{\eta,\kappa}}$ by \eqref{kap0}, \eqref{literki1}
and \eqref{literki2}. Arguing as in the proof of (ii),
we see that the systems $\{\mu_v\}_{v\in
V_{\eta,\kappa}}$ and $\{\varepsilon_v\}_{v\in
V_{\eta,\kappa}}$ with $\varepsilon_v\equiv 0$ satisfy
the assumptions of Theorem \ref{main}, and so $\slam$
is subnormal.
   \end{proof}
It turns out that conditions (ii) and (iii) of Theorem
\ref{omega2} are equivalent without assuming that
\eqref{zgod0} is satisfied.
   \begin{lem} \label{IBJ}
Let $\slam$ be a weighted shift on
$\tcal_{\eta,\kappa}$ with nonzero weights $\lambdab =
\{\lambda_v\}_{v \in V_{\eta,\kappa}^\circ}$ such that
$e_0 \in \dzn{\slam}$ and let $\{\mu_i\}_{i=1}^\eta$
be a sequence of Borel probability measures on
$\rbb_+$. Then conditions {\em (ii)} and {\em (iii)}
of Theorem {\em \ref{omega2}} $($with the same
$\kappa$$)$ are equivalent.
   \end{lem}
   \begin{proof}
(ii)$\Rightarrow$(iii) Let $\{\mu_{-l}\}_{l=0}^\kappa$
be the Borel probability measures on $\rbb_+$ defined
by \eqref{literki1}, \eqref{literki2} and
\eqref{literki3} with $\varepsilon_{-\kappa}$ given by
\eqref{literki4}. Set $\nu=\mu_{-\kappa}$. It follows
from \eqref{literki3} that for every $n \in J_\kappa$,
   \begin{align}  \label{intnu}
\int_\sigma s^n \D \nu(s) =
\Big|\prod_{j=0}^{\kappa-1} \lambda_{-j}\Big|^2 \;
\sum_{i=1}^\eta |\lambda_{i,1}|^2 \int_\sigma
\frac{1}{s^{\kappa + 1 - n}} \D \mu_i(s), \quad \sigma
\in \borel{\rbb_+}.
   \end{align}
This immediately implies \eqref{prob'}. By
\eqref{literki1}, \eqref{literki2} and \eqref{intnu},
we have
   \begin{align*}
\int_\sigma s^n \D \nu(s) =
   \begin{cases}
\cfrac{|\prod_{j=0}^{\kappa-1}
\lambda_{-j}|^2}{|\prod_{j=0}^{\kappa-n-1}
\lambda_{-j}|^2} \, \mu_{-(\kappa-n)}(\sigma) & \text{
if } n \in J_{\kappa-1},
   \\[3.5ex]
|\prod_{j=0}^{\kappa-1} \lambda_{-j}|^2 \,
\mu_{0}(\sigma) & \text{ if } n=\kappa,
   \end{cases}
   \end{align*}
for all $\sigma \in \borel{\rbb_+}$. Substituting
$\sigma=\rbb_+$ and using the fact that
$\{\mu_{-l}\}_{l=0}^{\kappa-1}$ are probability
measures, we obtain \eqref{prob}.

(iii)$\Rightarrow$(ii) Given $n \in J_\kappa$, we
define the positive Borel measure $\rho_n$ on $\rbb_+$
by $\rho_n(\sigma) = \int_\sigma s^n \D \nu(s)$ for
$\sigma \in \borel{\rbb_+}$. By \eqref{prob'},
equality \eqref{intnu} holds for $n=\kappa$. If this
equality holds for a fixed $n \in J_\kappa \setminus
\{1\}$, then $\rho_n(\{0\})=0$ and consequently
   \begin{align*}
\int_\sigma s^{n-1} \D \nu(s) = \int_\sigma
\frac{1}{s} \D \rho_n(s) \overset{\eqref{intnu}}=
\Big|\prod_{j=0}^{\kappa-1} \lambda_{-j}\Big|^2 \;
\sum_{i=1}^\eta |\lambda_{i,1}|^2 \int_\sigma
\frac{1}{s^{\kappa + 1 - (n-1)}} \D \mu_i(s)
   \end{align*}
for all $\sigma \in \borel{\rbb_+}$. Hence, by reverse
induction on $n$, \eqref{intnu} holds for all $n\in
J_\kappa$. Substituting $\sigma=\rbb_+$ into
\eqref{intnu} and using \eqref{prob}, we obtain
\eqref{zgod'} and \eqref{widly1}. It follows from
\eqref{intnu}, applied to $n=1$, that for every
$\sigma \in \borel{\rbb_+}$,
   \begin{multline}  \label{nusig}
\nu(\sigma) = \nu(\sigma \setminus \{0\}) + \nu(\{0\})
\delta_0(\sigma) = \int_\sigma \frac{1}{s} \D
\rho_1(s) + \nu(\{0\}) \delta_0(\sigma)
   \\
\overset{\eqref{intnu}}= \Big|\prod_{j=0}^{\kappa-1}
\lambda_{-j}\Big|^2 \; \sum_{i=1}^\eta
|\lambda_{i,1}|^2 \int_\sigma \frac{1}{s^{\kappa + 1}}
\D \mu_i(s) + \nu(\{0\}) \delta_0(\sigma).
   \end{multline}
Substituting $\sigma=\rbb_+$ into \eqref{nusig} and
using the fact that $\nu(\rbb_+)=1$, we obtain
\eqref{widly1'}. This completes the proof.
   \end{proof}
Under the assumption of determinacy, the sufficient
conditions for subnormality appearing in Theorem
\ref{omega2} become necessary (see also Remark
\ref{deterrem} below).
   \begin{thm}  \label{deter}
   Let $\slam$ be a subnormal weighted shift on
$\tcal_{\eta,\kappa}$ with nonzero weights $\lambdab =
\{\lambda_v\}_{v \in V_{\eta,\kappa}^\circ}$. If $e_0
\in \dzn{\slam}$ and
   \begin{align}    \label{detn+1}
\text{$\Big\{\sum_{i=1}^\eta \Big|\prod_{j=1}^{n+1}
\lambda_{i,j}\Big|^2\Big\}_{n=0}^\infty$ is a determinate
Stieltjes moment sequence,}
   \end{align}
then the following four assertions hold\/{\em :}
   \begin{enumerate}
   \item[(i)] if $\kappa = 0$, then there exists a
sequence $\{\mu_i\}_{i=1}^\eta$ of Borel probability
measures on $\rbb_+$ that satisfy \eqref{zgod0} and
\eqref{zgod},
   \item[(ii)] if $0 < \kappa < \infty$, then there
exists a sequence $\{\mu_i\}_{i=1}^\eta$ of Borel
probability measures on $\rbb_+$ that satisfy
\eqref{zgod0}, \eqref{zgod'}, \eqref{widly1} and
\eqref{widly1'},
   \item[(iii)] if $0 < \kappa < \infty$, then there
exist a sequence $\{\mu_i\}_{i=1}^\eta$ of Borel
probability measures on $\rbb_+$ and a Borel
probability measure $\nu$ on $\rbb_+$ that satisfy
\eqref{zgod0}, \eqref{prob} and \eqref{prob'},
   \item[(iv)] if $\kappa=\infty$, then there
exists a sequence $\{\mu_i\}_{i=1}^\eta$ of Borel
probability measures on $\rbb_+$ that satisfy
\eqref{zgod0}, \eqref{zgod'} and \eqref{widly1}.
   \end{enumerate}
Moreover, if\;\footnote{\;\label{stopa}Equivalently,
by \eqref{detn+2}, $e_0$ is a quasi-analytic vector of
$\slam$ (see \cite{nuss} for the definition).}
$\sum_{n=1}^\infty \big(\sum_{i=1}^\eta
\big|\prod_{j=1}^{n}
\lambda_{i,j}\big|^2\big)^{-\nicefrac{1}{2n}} =
\infty$, then \eqref{detn+1} is satisfied.
   \end{thm}
   \begin{proof}
It is easily seen that
   \begin{align} \label{detn+2}
\|\slam^{n+1} e_0\|^2 = \sum_{i=1}^\eta
\big|\prod_{j=1}^{n+1} \lambda_{i,j}\big|^2, \quad n
\in \zbb_+.
   \end{align}
By \cite[Proposition 4.1.1]{b-j-j-sA}, for every $u
\in V_{\eta,\kappa}$ the sequence $\{\|\slam^n
e_u\|^2\}_{n=0}^\infty$ is a Stieltjes moment
sequence. For each $i \in J_\eta$, we choose a
representing measure $\mu_i$ of $\{\|\slam^{n}
e_{i,1}\|^2\}_{n=0}^\infty$. It is easily seen that
\eqref{zgod0} holds. Since, by \eqref{detn+1} and
\eqref{detn+2}, the Stieltjes moment sequence
$\{\|\slam^{n+1} e_{0}\|^2\}_{n=0}^\infty$ is
determinate, we infer from \cite[Lemma
4.1.3]{b-j-j-sA}, applied to $u=0$, that \eqref{zgod}
holds and $\{\|\slam^{n} e_{0}\|^2\}_{n=0}^\infty$ is
a determinate Stieltjes moment sequence with the
representing measure $\mu_0$ given by
   \begin{align} \label{muu+2}
\mu_0(\sigma) = \sum_{i=1}^\eta |\lambda_{i,1}|^2
\int_\sigma \frac 1 s \D \mu_i(s) + \varepsilon_0
\delta_0(\sigma), \quad \sigma \in \borel{\rbb_+},
      \end{align}
where $\varepsilon_0$ is a nonnegative real number. In
view of the above, assertion (i) is proved.

Suppose $0 < \kappa \Le \infty$. Since $\{\|\slam^{n}
e_{0}\|^2\}_{n=0}^\infty$ is a determinate Stieltjes
moment sequence, we deduce from Lemma \ref{charsub-1},
applied to $u=-1$, that $\{\|\slam^{n+1}
e_{-1}\|^2\}_{n=0}^\infty$ and $\{\|\slam^{n}
e_{-1}\|^2\}_{n=0}^\infty$ are determinate Stieltjes
moment sequences~ and
   \begin{align} \label{jabko}
& \int_0^\infty \frac 1 s \D \mu_0(s) \Le
\frac{1}{|\lambda_{0}|^2},
   \\
& \mu_{-1}(\sigma) = |\lambda_{0}|^2 \int_\sigma \frac
1 s \D \mu_0(s) + \varepsilon_{-1} \delta_0(\sigma),
\quad \sigma \in \borel{\rbb_+}, \label{jabko2}
   \end{align}
where $\mu_{-1}$ is the representing measure of
$\{\|\slam^{n} e_{-1}\|^2\}_{n=0}^\infty$ and
$\varepsilon_{-1}$ is a nonnegative real number.
Inequality \eqref{jabko} combined with equality
\eqref{muu+2} implies that $\varepsilon_0=0$ and
therefore that \eqref{widly1'} holds for $\kappa=1$.
Substituting $\sigma=\rbb_+$ into \eqref{muu+2}, we
obtain \eqref{zgod'}. This completes the proof of
assertion (ii) for $\kappa=1$. Note also that
equalities \eqref{muu+2} and \eqref{jabko2}, combined
with $\varepsilon_0=0$, yield
   \begin{align*}
\mu_{-1}(\sigma) = |\lambda_{0}|^2 \sum_{i=1}^\eta
|\lambda_{i,1}|^2 \int_\sigma \frac 1 {s^2} \D
\mu_i(s) + \varepsilon_{-1} \delta_0(\sigma), \quad
\sigma \in \borel{\rbb_+}.
   \end{align*}
If $\kappa > 1$, then arguing by induction, we
conclude that for every $k\in J_{\kappa}$ the
Stieltjes moment sequences $\{\|\slam^{n+1}
e_{-k}\|^2\}_{n=0}^\infty$ and $\{\|\slam^{n}
e_{-k}\|^2\}_{n=0}^\infty$ are determinate and
   \begin{align} \label{mu-l}
\mu_{-l}(\sigma) =
\Big|\prod_{j=0}^{l-1}\lambda_{-j}\Big|^2
\sum_{i=1}^\eta |\lambda_{i,1}|^2 \int_\sigma \frac 1
{s^{l+1}} \D \mu_i(s), \quad \sigma \in
\borel{\rbb_+}, \, l\in J_{\kappa - 1},
   \end{align}
where $\mu_{-l}$ is the representing measure of
$\{\|\slam^{n} e_{-l}\|^2\}_{n=0}^\infty$.
Substituting $\sigma=\rbb_+$ into \eqref{mu-l}, we
obtain \eqref{widly1}. This completes the proof of
assertion (iv). Finally, if $1 < \kappa < \infty$,
then again by Lemma \ref{charsub-1}, now applied to
$u=-\kappa$, we have $\int_0^\infty \frac 1 s \D
\mu_{-\kappa+1}(s) \Le
\frac{1}{|\lambda_{-\kappa+1}|^2}$. This inequality
together with \eqref{mu-l} yields \eqref{widly1'},
which completes the proof of assertion (ii).

(iii) can be deduced from (ii) via Lemma \ref{IBJ}.

To show the ``moreover'' part, we can argue as in the
proof of \cite[Theorem 5.3.1]{b-j-j-sA} (see also
Footnote \ref{stopa}). This completes the proof.
   \end{proof}
   \begin{rem}  \label{deterrem}
A careful look at the proof reveals that Theorem
\ref{deter} remains valid if instead of assuming that
$\slam$ is subnormal, we assume that $u$ is a
Stieltjes vertex for every $u \in \big\{-k\colon k\in
J_\kappa\big\} \sqcup \{0\} \sqcup \dzi{0}$.
   \end{rem}
   \begin{cor}   \label{hakol}
   Let $\slam$ be a weighted shift on
$\tcal_{\eta,\kappa}$ with nonzero weights $\lambdab =
\{\lambda_v\}_{v \in V_{\eta,\kappa}^\circ}$ such that
$e_0 \in \dzn{\slam}$. Suppose that $v$ is a Stieltjes
vertex for every $v \in \big\{-k\colon k\in
J_\kappa\big\} \sqcup \{0\} \sqcup \dzi{0}$ and that
$\{\|\slam^{n+1} e_0\|^2\}_{n=0}^\infty$ is a
determinate Stieltjes moment sequence. Then the
following three assertions hold\/{\em :}
   \begin{enumerate}
   \item[(i)] $\slam$ is subnormal,
   \item[(ii)] $\{\|\slam^{n+1}
e_{-j}\|^2\}_{n=0}^\infty$ is a determinate Stieltjes
moment sequence for every integer $j$ such that $0 \Le
j \Le \kappa$,
   \item[(iii)] $\slam$  satisfies the consistency
condition \eqref{alanconsi} at the vertex $u=-j$ for
every integer $j$ such that $0 \Le j \Le \kappa$.
   \end{enumerate}
   \end{cor}
   \begin{proof}
   (i) In view of Remark \ref{deterrem}, there exists
a sequence $\{\mu_i\}_{i=1}^\eta$ of Borel probability
measures on $\rbb_+$ satisfying \eqref{zgod0} and one
of the conditions (i), (ii) and (iv) of Theorem
\ref{omega2}. Hence, by Theorem \ref{omega2}, $\slam$
is subnormal.

   (ii) See the proof of Theorem \ref{deter}.

   (iii) Apply (ii) and \cite[Lemma
4.1.3\,(ii)]{b-j-j-sA}.
   \end{proof}
   \begin{rem}
Note that the consistency condition \eqref{alanconsi}
at a vertex $u$ depends on the choice of the system
$\{\mu_v\}_{v\in\dzi{u}}$ of representing measures
$\mu_v$ of Stieltjes moment sequences $\{\|\slam^n
e_v\|^2\}_{n=0}^\infty$. However, by \cite[Lemma
4.1.3]{b-j-j-sA}, if the Stieltjes moment sequence
$\{\|\slam^{n+1} e_u\|^2\}_{n=0}^\infty$ is
determinate, then the consistency condition
\eqref{alanconsi} at $u$ is independent of the choice
of $\{\mu_v\}_{v\in\dzi{u}}$, i.e., it is satisfied
for every system $\{\mu_v\}_{v\in\dzi{u}}$ of
representing measures $\mu_v$ of Stieltjes moment
sequences $\{\|\slam^n e_v\|^2\}_{n=0}^\infty$. This
and assertion (ii) of Corollary \ref{hakol} justify
not mentioning explicitly representing measures in
assertion (iii) of this corollary.
   \end{rem}
   \section{\label{sec-in}Subnormality via subtrees}
Proposition \ref{subtree} below shows that the study
of subnormality of weighted shifts on rootless
directed trees can be reduced, in a sense, to the case
of directed trees with root. Theorem \ref{main}, which
is our main criterion for subnormality of weighted
shifts on directed trees, seems to be inapplicable in
this context. However, we can employ an inductive
limit approach.
   \begin{thm} \label{tw1+1}
Let $S$ be a densely defined operator in a complex
Hilbert space $\hh$. Suppose that there are a family
$\{\hh_\omega\}_{\omega \in \varOmega}$ of closed
linear subspaces of $\hh$ and an upward directed
family $\{\xx_\omega\}_{\omega \in \varOmega}$ of
subsets of $\hh$ such that
   \begin{enumerate}
   \item[(i)] $\xx_\omega \subseteq
\dzn{S}$ and $S^n(\xx_\omega) \subseteq \hh_\omega$
for all $n\in \zbb_+$ and $\omega \in \varOmega$,
   \item[(ii)] $\ff_\omega:=\lin \bigcup_{n=0}^\infty
S^n(\xx_\omega)$ is dense in $\hh_\omega$ for every
$\omega \in \varOmega$,
   \item[(iii)]  $S|_{\ff_\omega}$ is  a
subnormal operator in $\hh_\omega$ for every $\omega
\in \varOmega$,
   \item[(iv)]  $\ff:=\lin \bigcup_{n=0}^\infty
S^n\big(\bigcup_{\omega \in \varOmega}
\xx_\omega\big)$ is a core of $S$.
   \end{enumerate}
Then $S$ is subnormal.
   \end{thm}
   \begin{proof}
The families $\{\ff_\omega\}_{\omega \in \varOmega}$
and $\{\hh_\omega\}_{\omega \in \varOmega}$ are upward
directed, $S(\ff_\omega) \subseteq \ff_\omega$ for all
$\omega \in \varOmega$, $\ff = \bigcup_{\omega \in
\varOmega} \ff_\omega$ and $S(\ff) \subseteq \ff$.
Hence, we can argue as in the proof of \cite[Theorem
3.1.1]{b-j-j-sA} by using \cite[Theorem 21]{c-s-sz}.
   \end{proof}
Let $\slam$ be a weighted shift on a directed tree
$\tcal$ with weights $\lambdab=\{\lambda_v\}_{v\in
V^\circ}$. Note that if $u \in V$, then the space
$\ell^2(\des{u})$ (which is regarded as a closed
linear subspace of $\ell^2(V)$) is invariant for
$\slam$, i.e.,
   \begin{align} \label{ilb}
\slam\big(\dz{\slam} \cap \ell^2(\des{u})\big)
\subseteq \ell^2(\des{u}).
   \end{align}
(For this, apply \eqref{lamtauf} and the inclusion
$\paa\big(V\setminus \big(\des{u} \cup
\Ko{\tcal}\big)\big) \subseteq V\setminus \des{u}$.)
Denote by $\slam|_{\ell^2(\des{u})}$ the operator in
$\ell^2(\des{u})$ given by
$\dz{\slam|_{\ell^2(\des{u})}} = \dz{\slam} \cap
\ell^2(\des{u})$ and $\slam|_{\ell^2(\des{u})}f =
\slam f$ for $f \in \dz{\slam|_{\ell^2(\des{u})}}$. It
is easily seen that $\slam|_{\ell^2(\des{u})}$
coincides with the weighted shift on the rooted
directed tree $(\des{u}, (\des{u}\times \des{u}) \cap
E)$ with weights $\{\lambda_v\}_{v \in
\des{u}\setminus \{u\}}$ (see \cite[Proposition
2.1.8]{j-j-s} for more details on this and related
subtrees).

   We are now ready to apply the aforementioned
inductive limit procedure to prove the unbounded
counterpart of \cite[Corollary 6.1.4]{j-j-s}.
   \begin{pro}\label{subtree} Let $\slam$ be a weighted
shift on a rootless directed tree $\tcal$ with weights
$\lambdab=\{\lambda_v\}_{v\in V^\circ}$. Suppose that
$\escr \subseteq \dzn{\slam}$. If $\varOmega$ is a
subset of $V$ such that $V=\bigcup_{\omega\in
\varOmega} \des{\omega}$, then the following
conditions are equivalent{\em :}
   \begin{enumerate}
   \item[(i)] $\slam$ is subnormal,
   \item[(ii)] for every $\omega\in \varOmega$,
$\slam|_{\ell^2(\des{\omega})}$ is subnormal as an
operator acting in $\ell^2(\des{\omega})$.
   \end{enumerate}
   \end{pro}
   \begin{proof}
(ii)$\Rightarrow$(i) Using an induction argument and
\eqref{ilb} one can show that $\slam^n e_v \in
\ell^2(\des{v}) \subseteq \ell^2(\des{u})$ for all
$n\in \zbb_+$, $v \in \des{u}$ and $u \in V$. Hence
   \begin{align*}
\xx_\omega := \lin\big\{e_v\colon v \in
\des{\omega}\big\} \subseteq \dzn{\slam} \text{ and }
\slam^n(\xx_\omega) \subseteq \ell^2(\des{\omega})
   \end{align*}
for all $\omega \in \varOmega$ and $n\in \zbb_+$. It
follows from \cite[Proposition 2.1.4]{j-j-s} and the
equality $V=\bigcup_{\omega\in \varOmega}
\des{\omega}$ that for each pair $(\omega_1,\omega_2)
\in \varOmega \times \varOmega$, there exists $\omega
\in \varOmega$ such that $\des{\omega_1} \cup
\des{\omega_2} \subseteq \des{\omega}$, and thus
$\{\xx_\omega\}_{\omega \in \varOmega}$ is an upward
directed family of subsets of $\ell^2(V)$. By applying
\cite[Proposition 3.1.3]{j-j-s} and Theorem
\ref{tw1+1} to $S=\slam$ and $\hh_\omega=
\ell^2(\des{\omega})$, we deduce that $\slam$ is
subnormal.

The reverse implication (i)$\Rightarrow$(ii) is
obvious because $\xx_\omega \subseteq
\dz{\slam|_{\ell^2(\des{\omega})}}$.
   \end{proof}
It follows from \cite[Proposition 2.1.6]{j-j-s} that if
$\tcal$ is a rootless directed tree, then
$V=\bigcup_{k=1}^\infty \des{\paa^k(u)}$ for every $u
\in V$, and so the set $\varOmega$ in Proposition
\ref{subtree} may always be chosen to be countably
infinite.
   \subsection*{Acknowledgements} A substantial part of
this paper was written while the second and the fourth
authors visited Kyungpook National University during
the autumn of 2010 and the spring of 2011. They wish
to thank the faculty and the administration of this
unit for their warm hospitality.
   \bibliographystyle{amsalpha}
   
   \end{document}